\documentclass[a4paper,11pt]{article}
\usepackage[utf8]{inputenc}
\usepackage[T1]{fontenc}

\usepackage{amsthm,amsmath}
\usepackage{tikz}
\usepackage{mathrsfs,amssymb,amsfonts} 
\usepackage{enumitem}
\usepackage{fullpage}
\usepackage{hyperref, enumerate}
\usepackage[babel]{microtype}
\usepackage[english]{babel}
\usepackage[capitalise]{cleveref}

\usepackage{thmtools}
\usepackage{mathtools, comment}
\usepackage{amssymb}
\usepackage[nomath]{lmodern}
\usepackage{graphicx}
\usepackage{pgf,tikz,tkz-graph,subcaption}
\usetikzlibrary{arrows,shapes}
\usetikzlibrary{decorations.pathreplacing}
\usepackage{tkz-berge}
\usepackage{enumitem}
\usepackage[normalem]{ulem}
\usepackage{hyperref}
\hypersetup{colorlinks = true, linkcolor = blue, citecolor = blue, urlcolor = blue}

\allowdisplaybreaks

\usepackage[margin=1in]{geometry}
\newtheorem{defi}{Definition}

\newtheorem{thm}[defi]{Theorem}

\newtheorem{remark}[defi]{Remark}
\newtheorem{claim}[defi]{Claim}
\newcommand*{\myproofname}{Proof}
\newenvironment{claimproof}[1][\myproofname]{\begin{proof}[#1]}{\end{proof}}

\DeclareMathOperator{\lcm}{lcm}

\title{Resolution of Erd\H{o}s' problems about unimodularity}
\author{Stijn Cambie \thanks{Department of Computer Science, KU Leuven Campus Kulak-Kortrijk, 8500 Kortrijk, Belgium. Supported by a postdoctoral fellowship by the Research Foundation Flanders (FWO) with grant number 1225224N. Email: \protect\href{mailto:stijn.cambie@hotmail.com}{\protect\nolinkurl{stijn.cambie@hotmail.com}}}}

\begin{document}
\parindent=0cm
\maketitle

\begin{abstract}
    Letting $\delta_1(n,m)$ be the density of the set of integers with exactly one divisor in $(n,m)$, Erd\H{o}s wondered if $\delta_1(n,m)$ is unimodular for fixed $n$.
    We prove this is false in general, as the sequence $(\delta_1(n,m))$ has superpolynomially many local extrema. However, we confirm unimodality in the single case for which it occurs; $n = 1$.
    We also solve the question on unimodality of the density of integers whose $k^{th}$ prime is $p$.
\end{abstract}

\section{Introduction}\label{sec:intro}

In this note, we address the questions $690$ and $692$ from ~\url{https://www.erdosproblems.com} (\cite[page.~75, 78]{Er79e}), both in the negative.

Let $\delta_1(n,m)$ (resp. $\delta_r(n,m)$) be the density of the set of integers with exactly one (resp. $r$) divisor in $(n,m)=\{n+1,n+2, \ldots, m-1\}.$ 
 Erd\H{o}s wondereded if $\delta_1(n,m)$ is unimodular for fixed $n$, i.e., if $(\delta_1(n,m)_{m \ge n+2}$ has at most one local maximum.
With a computer program (\cite[doc. \texttt{Erdosproblem692\_n\_le20}]{C25github_EP692}), one can check that $(\delta_1(n,m))_{m \ge n+2}$ is not unimodular for small $n$ ($2\le n \le 20$), answering the question.

As an example, using the principle of inclusion-exclusion and a recursion, with $\Pr(r\mid x)$ denoting the probality that a random integer $x$ is a multiple of $r$, one can verify by hand that 
\begin{align*}
    \delta_1(3,6)&=\Pr(4 \mid x )+ \Pr(5 \mid x )-2\cdot \Pr(4,5 \mid x )  = \frac 14+\frac 15 - 2\cdot  \frac{1}{20}&&=\frac{7}{20} =0.35\\
    \delta_1(3,7)&= \frac 14 + \frac 15+\frac 16-2\left( \frac{1}{20} + \frac{1}{12}+\frac1{30} \right) +3\cdot \frac{1}{60}&&=\frac 13 \sim 0.33\\
    \delta_1(3,8)&=\frac 67 \cdot \delta_1(3,7)+\frac 17\cdot  \delta_0(3,7)=\frac 67\cdot \frac 13 + \frac 17 \cdot \frac{8}{15}&&= \frac{38}{105} \sim 0.36
\end{align*}
    
Since $\delta_1(3,7)< \min\{\delta_1(3,6), \delta(3,8) \}$ the sequence $( \delta_1(3,m) )_{m \ge 5}$ is not unimodal, from which one concludes. 
Inspired by communication with Thomas Bloom, we address the question more precisely.
We prove that the sequence has superpolynomially many local maxima in general, and is thus very far from being unimodal. 
Nonetheless, there is still one case in which the possible intuition about nice behaviour is true. When $n=1$, the sequence $(\delta_1(1,m))_m$ of densities of the set of integers with exactly one non-trivial divisor bounded by $m$ is unimodal, being non-increasing. The proofs can be found  in~\cref{sec:main}.

As a second result, let $d_k(p)$ be the density of integers whose $k^{th}$ prime is $p$.
Erd\H{o}s could not disprove unimodularity of the sequence $(d_k(p))_p$. We show that unimodularity is true for $k\le 3$ and give counterexamples for $k>3.$ This is done in~\cref{sec:main690}.


In the proofs, we will assume the reader is familiar with Landau notation $O(), \omega(), \Theta()$ for functions that are bounded from above by a multiple of an other function, by below, and by both respectively.

\section{Main proofs for problem $692$}\label{sec:main}

We first prove the one case for which Erd\H{o}s' question has a positive answer, since it is the more elementary result.

\begin{thm}
    The sequence $\delta_1(1,m)$ is non-increasing (and thus unimodular) in $m$.
\end{thm}

\begin{proof}
    For fixed $m \in \mathbb N$, let the small primes be $2\le p_1, p_2, \ldots  p_r \le \sqrt{m-1}$ and the larger primes be $\sqrt{m}\le q_1, q_2, \ldots, q_s \le m-1.$
    Let $L= \prod_{i=1}^r p_i^2 \cdot \prod_{i=1}^s q_i.$
    Let $\varphi(L)$ and $A$ be the number of integers in $\{1,2,\ldots, L\}$ which have, respectively, zero or exactly one divisor in $\{2,3,\ldots, m-1\}.$
    Note that $A$ and $L$ are functions of $m$, but we won't write $m$ explicitly for ease of notation, and $\varphi$ is the Euler totient function
    
    Now $\delta_1(1,m)=\frac{A}{L},$ since a number has exactly one divisor in $\{2,3,\ldots, m-1\}$ if it is a multiple of exactly one prime in $\{p_1, \ldots, p_r, q_1, \ldots, q_s\}$ and not a multiple of $p_i^2$ for any $1 \le i \le r$, and it is thus only dependent on its residue modulo $L.$ Similarly $\delta_0(n,m)=\frac{ \varphi{(L)}}{L}.$
    
    We will prove two properties in parallel;
    $\frac{A}{L}$ is non-increasing in $m$ and $A \ge \varphi(L)$ for every $m \ge 3.$

    In the base case, where $m=3$, we have $L=2$ and $A=\varphi(L)=1$ (half of the integers are a multiple of $2$ and half of them are not).

    In the induction step, we only have to consider $m-1$ equal to a prime, or the square of a prime.
    Let $L,A$ be the values for $m-1$ and $A',L'$ the values for $m.$

    \textbf{Case $m-1=p$:} Compared with $m-1$, for $m$ we find that 
    $L'=pL,$
    $\varphi(L')=(p-1)\varphi(L)$ and
    $A'=A(p-1)+ \varphi(L)$.
    The latter since for every residue modulo $L$, there are $p-1$ possibilities modulo $pL$ that are not a multiple of $p,$ and one which is a multiple of $p.$

    Now \begin{equation}\label{eq:p}\frac{A'}{\varphi(L')}=\frac{A}{\varphi(L)}+\frac{1}{p-1}\end{equation} and 
    $$\frac{A'}{L'}=\frac{A(p-1)+ \varphi(L)}{pL}\le \frac{A}{L}$$ since $\varphi(L) \le A$ by the induction hypothesis.

    \textbf{Case $m-1=p^2$:} Compared with $m-1$, for $m$ we now find that 
    $L'=pL,$
    $\varphi(L')=p\varphi(L)$ and
    $A'=Ap-\frac{\varphi(L)}{p-1}$, implying immediately that $\frac{A'}{L'}<\frac{A}{L}.$
    The latter since for every integer $0<x\le L$ that is not a multiple of $p$ that has one divisor among $\{2,3,\ldots, m-2\},$ each integer of the form $iL+x$ with $0 \le i \le p-1$ has the same property.
    For an integer $0 \le x \le L$ that only has the divisor $p$ among $\{2,3,\ldots, m-2\},$ there are $p-1$ choices for $0 \le i \le p-1$ such that $p^2 \nmid iL+x.$ Here there are $\varphi(L/p)=\frac{\varphi(L)}{p-1}$ choices for $x$, since $x=px'$ where $x' \le \frac Lp$ and $x'$ is relative prime with the other primes less than $p^2$ and thus $\frac{L}{p}$.
   We further have that \begin{equation}\label{eq:p^2}
       \frac{A'}{\varphi(L')}=\frac{A}{\varphi(L)}-\frac{1}{p(p-1)}.
   \end{equation}

Since $\frac{1}{3-1}=\frac{1}{2(2-1)}$ and $\frac{1}{5-1}>\frac1{3(3-1)}+\frac1{5(5-1)}$, from \cref{eq:p} and~\cref{eq:p^2} we deduce that $\frac{A}{\varphi(L)} \ge 0$ and the latter is strict once $m \ge 6.$

We conclude that both statements are true by induction, and thus $\delta_1(1,m)$ is indeed non-increasing.
\end{proof}

\begin{remark}
    The above recursions can also be implemented to compute 
 $\delta_1(n,m)$ efficiently for small $n$, leading to a linear time program as a function of $m$. This has been done for $n \in \{2,3\}$ in~\cite[doc. \texttt{Recurs\_n2} and \texttt{Recurs\_n3}]{C25github_EP692}
\end{remark}

Next, we prove there are superpolynomially many  local maxima.


\begin{thm}
    For some fixed $c>0$, the sequence $(\delta_1(n,m))_{m \ge n+2}$ contains $\omega(\exp(n^c))$ many local maxima.
\end{thm}

\begin{proof}
We start proving the following claim, which we will apply later.
 \begin{claim}\label{clm:delta0>delta1}
        There exists $c>0$ such that for every sufficiently large $n$ and $m= \Theta( \exp(3n^c)),$ $\delta_0(n,m+1)>\delta_1(n,m+1).$
    \end{claim}
    \begin{claimproof}
         Let $L=\lcm\{n+1,n+2, \ldots, m\}=\lcm\{1,2,\ldots, m\}.$
    By the definition of Euler's totient function and by Mertens (third) theorem~\cite{Mertens78}
    $\frac{ \varphi(L)}{L} = \prod_{p \le m} \frac{p-1}{p} \sim exp(-\gamma) \frac{1}{\log(m)} $, where $exp(\gamma)<2$ is a constant. As a corollary, we have $\frac{ \varphi(L)}{L}> \frac{1}{2\log(m)}.$
    A classical estimate of the harmonic numbers says $H_n =\sum_{i=1}^n \frac 1i > \log(n)$. 
       Using these two inequalities, we derive that
  $$\delta_0(n,m)= \frac{ \sum_{i=1}^n \varphi\left( \frac L i \right) }{L} \ge \frac{ \sum_{i=1}^n  \frac {\varphi (L)} i  }{L}\ge \frac{ H_n \cdot \varphi(L) }{L} > \frac{ \log(n)}{2 \log(m)}.$$

By~\cite[Thm.~4]{Ford08},
$$\delta_1(n,m)= O \left(  \frac{ \log \log (m/n) }{\log(m/n)} \right).$$
Assuming $c$ is chosen sufficiently small,
we conclude $\delta_1(n,m)<\delta_0(n,m)$.
\qedhere 
    \end{claimproof}
    We will prove that $\delta_1(n,p+1)>\delta_1(n,p)$ for the primes $p$ with $p= \Theta( \exp(3n^c))$ and
    $\delta_1(n,2p+1)<\delta_1(n,2p)$ for every prime $p>n.$

    The result then follows from taking the longest sequences $(p_i)_i$ and $(q_i)_i$ of primes satisfying $\exp(3n^c)<p_1<2q_1<p_2<2q_2<p_3<\ldots<p_r<2q_r<2\exp(3n^c).$
    By a result on prime gaps~\cite{BHP01}, we know $r=\omega(\exp(0.475\cdot 3n^c)))$.

    To prove that $\delta_1(n,p+1)>\delta_1(n,p)$,
    note that $\delta_1(n,p+1)=\frac{p-1}{p} \delta_1(n,p) + \frac{1}{p} \delta_0(n,p),$ and this is larger than $\delta_1(n,p)$ by~\cref{clm:delta0>delta1} for $p=\Theta( \exp(3n^c))$.

    The inequality $\delta_1(n,2p+1)<\delta_1(n,2p)$ for a prime $p>n$ is almost trivial.
    Since multiples of $2p$ are multiples of $p$, there are no numbers whose only divisor in $(n,2p+1)$ is $2p.$ 
    On the other hand, there are multiples of $2p$ (this only depends on the residue modulo $\lcm\{n+1,\ldots,2p\}$) that have only one divisor in $(n,2p),$ but with two divisors in $(n,2p+1).$ 
    \qedhere
\end{proof}

\section{Main proofs for problem $690$}\label{sec:main690}

In~\cref{sec:main}, we noted that unimodularity in problem $692$ was not true, due to the difference of extending the range with a prime or a composite number.
In problem $690$, we always extend with a prime and the fact that unimodality is not always true now comes from the irregularity of prime gaps.

Similar to the results in~\cref{sec:main}, there are only a few cases, $1\le k\le 3$, for which unimodularity is true.
We also compute that it is not true for $4 \le k \le 20.$ The computational results indicate that unimodality cannot be expected for any $k \ge 4$, but we did not bother proving this for all such $k$ (note that from the proof one can easily deduce that every sequence $d_k(p)$ is eventually decreasing).

\begin{thm}
    For every $k \in \{1,2,3\}$, the sequence $d_k(p)$ is unimodular.
    The sequence $d_k(p)$ is not unimodular for every $4 \le k \le 20.$
\end{thm}

\begin{proof}
    Let the consecutive primes be ordered as $p_0=2, p_1=3, \ldots$ 
    Let $\delta_r(i),$ the density of integers with exactly $r$ prime divisors among $\{p_0, p_1, \ldots, p_i\}.$
    Note that $\delta_r(0)=\frac{1}{2}$ if $r \in \{0,1\}$ and $\delta_r(0)=0$ if $r>1.$

    Now, we prove the following simple recursion.

    \begin{claim}\label{clm:recursion}
        $\delta_0(i)=\frac{p_i-1}{p_i}\delta_0(i-1)$ for every $i \ge 1$\\
        $\delta_r(i)=\frac{p_i-1}{p_i}\delta_r(i-1) +\frac 1p \delta_{r-1}(i-1)$ for every $r, i \ge 1$
    \end{claim}
    \begin{claimproof}
        Let $L'=\prod_{j=0}^i p_j$ and $L=\prod_{j=0}^{i-1} p_j$.
        The first inequality follows from considering the Euler's totient function on $L'$.

        For every number modulo $L$, there are $p_i-1$ choices that are not a multiple of $p_i$, and one of them which is.
        This implies that for every $x \in \{1,2,\ldots, L\}$ which has $r$ prime factors among $p_j, 0 \le j \le i-1$, there are $p_i-1$ choices for $x'\in \{1,2,\ldots, L'\}$ with $r$ prime factors satisfying $x'\equiv x \pmod L$.
        For every $x \in \{1,2,\ldots, L\}$ with $r-1$ prime factors bounded by $p_{i-1},$ there is one $x'\in \{1,2,\ldots, L'\}$ which is a multiple of $p_i$ and satisfies $x'\equiv x \pmod L$ by the Chinese remainder theorem.
        These reductions are also valid in the other direction.
    \end{claimproof}

    A simple corollary of~\cref{clm:recursion} is that if $\delta_{r-1}(i)$ is non-increasing for $i \ge i_0$ and 
    $\delta_r(i')<\delta_r(i'-1)$ for some $i'\ge i_0,$
    then $\delta_{r}(i)$ is non-increasing for $i \ge i'.$

    It is trivial that $\delta_0$ is a decreasing sequence.

    We have $\delta_1(0)=\delta_1(1)=\frac 12,$ and the sequence is further decreasing.

    For $\delta_2,$ we verified that the sequence is decreasing from $i=23$ onwards.

    Now since $d_k(p_i)=\frac{\delta_{k-1}(i-1)}{p_i},$ we deduce easily that $d_k$ for $k \in \{1,2,3\}$ is unimodular by checking the first $25$ values, see~\cite[doc. \texttt{690\_k<=20}]{C25github_EP692}.
    For this, note that $\delta_{k-1}(i-1)<\delta_{k-1}(i)$ implies $\frac{\delta_{k-1}(i-1)}{p_i}<\frac{\delta_{k-1}(i)}{p_{i+1}}$.

    For $4 \le k \le 20,$ it has been checked in~\cite[doc. \texttt{690\_k<=20}]{C25github_EP692}. A few computations confirming non-unimodularity are also listed in~\cref{sec:app1}.
\end{proof}




\section*{Appendix}\label{sec: appendix}

\appendix

\section{First terms of sequences for problem $690$}\label{sec:app1}

Using the recursions from~\cref{clm:recursion}, we can compute the first values of the sequence $\delta_k$ and $d_k$ for small $k$ and conclude.

\begin{align*}
    (\delta_0(i))_{0 \le i \le 9} &= \left(\frac{1}{2}, \frac{1}{3}, \frac{4}{15}, \frac{8}{35}, \frac{16}{77}, \frac{192}{1001}, \frac{3072}{17017}, \frac{55296}{323323}, \frac{110592}{676039}, \frac{442368}{2800733}\right)\\
    (\delta_1(i))_{0 \le i \le 9} &= \left(\frac{1}{2}, \frac{1}{2}, \frac{7}{15}, \frac{46}{105}, \frac{44}{105}, \frac{288}{715}, \frac{33216}{85085}, \frac{613248}{1616615}, \frac{151296}{408595}, \frac{391584768}{1078282205}\right)\\
    (\delta_2(i))_{0 \le i \le 9} &= \left(0, \frac{1}{6}, \frac{7}{30}, \frac{4}{15}, \frac{326}{1155}, \frac{628}{2145}, \frac{992}{3315}, \frac{98304}{323323}, \frac{125568}{408595}, \frac{733440}{2369851}\right)\\
    (\delta_3(i))_{0 \le i \le 9} &= \left(0, 0, \frac{1}{30}, \frac{13}{210}, \frac{31}{385}, \frac{206}{2145}, \frac{1308}{12155}, \frac{81544}{692835}, \frac{738544}{5870865}, \frac{61026496}{462120945}\right)\\
    (\delta_4(i))_{0 \le i \le 9} &= \left(0, 0, 0, \frac{1}{210}, \frac{23}{2310}, \frac{1}{65}, \frac{734}{36465}, \frac{336}{13585}, \frac{35272}{1225785}, \frac{103905392}{3234846615}\right)
\end{align*}

Next, we consider the subsequence $d_k(p)$ ranging over all primes between $p_0=2$ (or $p_1=3$) and $p_{10}=31$.

\begin{align*}
    (d_1(p))_{2 \le p \le p_{10}} &= \left( \frac12, \frac{1}{6}, \frac{1}{15}, \frac{4}{105}, \frac{8}{385}, \frac{16}{1001}, \frac{192}{17017}, \frac{3072}{323323}, \frac{55296}{7436429}, \frac{110592}{19605131}, \frac{442368}{86822723}\right)\\
    (d_2(p))_{3 \le p \le p_{10}} &= \left(\frac{1}{6}, \frac{1}{10}, \frac{1}{15}, \frac{46}{1155}, \frac{44}{1365}, \frac{288}{12155}, \frac{33216}{1616615}, \frac{613248}{37182145}, \frac{151296}{11849255}, \frac{391584768}{33426748355}\right)\\
    (d_3(p))_{3 \le p \le p_{10}} &= \left(0, \frac{1}{30}, \frac{1}{30}, \frac{4}{165}, \frac{326}{15015}, \frac{628}{36465}, \frac{992}{62985}, \frac{98304}{7436429}, \frac{125568}{11849255}, \frac{733440}{73465381}\right)\\
    (d_4(p))_{3 \le p \le p_{10}} &= \left(0, 0, \frac{1}{210}, \frac{13}{2310}, \frac{31}{5005}, \frac{206}{36465}, \frac{1308}{230945}, \frac{81544}{15935205}, \frac{738544}{170255085}, \frac{61026496}{14325749295}\right)\\
    (d_5(p))_{3 \le p \le p_{10}} &= \left(0, 0, 0, \frac{1}{2310}, \frac{23}{30030}, \frac{1}{1105}, \frac{734}{692835}, \frac{336}{312455}, \frac{35272}{35547765}, \frac{103905392}{100280245065}\right)
\end{align*}

The first three partial sequences are decreasing once initial zeros are removed, and thus unimodular.

For the fourth sequence, 
$ \frac{206}{36465} < \frac{31}{5005}, \frac{1308}{230945}$.
The fifth sequence is not unimodular since 
$\frac{35272}{35547765}< \frac{336}{312455}, \frac{103905392}{100280245065}$.

\end{document}